\documentclass[11pt,a4paper,swedish,english]{amsart}
\usepackage{babel}
\usepackage[square,numbers]{natbib}
\usepackage{url}

\usepackage[T1]{fontenc}
\usepackage[latin9]{inputenc}
\usepackage{float}
\usepackage{textcomp}
\usepackage{amsfonts,amsmath,amssymb,amsthm}
\usepackage{graphicx}
\usepackage{esint}

\makeatletter
\@ifundefined{textcolor}{}
{%
 \definecolor{BLACK}{gray}{0}
 \definecolor{WHITE}{gray}{1}
 \definecolor{RED}{rgb}{1,0,0}
 \definecolor{GREEN}{rgb}{0,1,0}
 \definecolor{BLUE}{rgb}{0,0,1}
 \definecolor{CYAN}{cmyk}{1,0,0,0}
 \definecolor{MAGENTA}{cmyk}{0,1,0,0}
 \definecolor{YELLOW}{cmyk}{0,0,1,0}
 }
\numberwithin{equation}{section}
\numberwithin{figure}{section}
\theoremstyle{plain}
\newtheorem{thm}{Theorem}
  \theoremstyle{definition}
  \newtheorem{defn}[thm]{Definition}
  \theoremstyle{plain}
  \newtheorem{lem}[thm]{Lemma}
  \theoremstyle{plain}
  \newtheorem{cor}[thm]{Corollary}

\makeatother

\begin{document}

\title{\noindent The Dynamical Functional Particle Method}

\author{Mårten Gulliksson, Sverker Edvardsson, and Andreas Lind}

\address{Division of Computational Mathematics and Physics, Mid Sweden University,
SE-85170 Sundsvall, Sweden}

\email{marten.gulliksson@miun.se,sverker.edvardsson@miun.se,andreas.lind@miun.se}
\begin{abstract}
We present a new algorithm which is named the Dynamical Functional
Particle Method, DFPM. It is based on the idea of formulating a finite
dimensional damped dynamical system whose stationary points are the
solution to the original equations. The resulting Hamiltonian dynamical
system makes it possible to apply efficient symplectic integrators.
Other attractive properties of DFPM are that it has an exponential
convergence rate, automatically includes a sparse formulation and
in many cases can solve nonlinear problems without any special treatment.
We study the convergence and convergence rate of DFPM. It is shown
that for the discretized symmetric eigenvalue problems the computational
complexity is given by $\mathcal{O}\left(N^{(d+1)/{d}}\right)$,
where \emph{d} is the dimension of the problem and \emph{N} is the
vector size. An illustrative example of this is made for the 2-dimensional
Schrödinger equation. Comparisons are made with the standard numerical
libraries ARPACK and LAPACK. The conjugated gradient method and shifted
power method are tested as well. It is concluded that DFPM is both
versatile and efficient.
\end{abstract}

\keywords{\texttt{Dynamical systems, Linear eigenvalue problems, ARPACK, Particle
methods, DFPM, Lyapunov function, Hamiltonian dynamics}}

\maketitle

\section{\textbf{Introduction}}

\subsection{The Dynamical Functional Particle Method}

The goal of this paper is to present an idea for solving equations
by formulating a dynamical system whose stationary solution is the
solution of the original equations. Examples of equations that can
be solved are ordinary and partial differential equations, linear
or nonlinear system of equations, and particularly eigenvalue problems.
In this section we begin by formulating the equation and the dynamical
system in an abstract setting. We then give the corresponding finite
dimensional formulation by discretizing the infinite dimensional problem.
This discretized problem will then be analyzed and studied throughout
the paper.

Let $\mathcal{F}$ be an operator and $v=v(x),\, v:X\rightarrow\mathbb{R}^{k},\, k\in\mathbb{N}$,
where $X$ is a Banach space that will be defined by the actual problem
setting. We consider the abstract equation\begin{equation}
\mathcal{F}(v)=0\label{Lu}\end{equation}
that could be, e.g., a differential equation. Further, a parameter
$t$ is introduced, interpreted as artificial time, which belongs
to the interval $T=[t_{0},\infty)$. A related equation in
 $u=u(x,t),\, u:X\times T\rightarrow\mathbb{R}^{k}$
is formulated as
\begin{equation}
\mu u_{tt}+\eta u_{t}=\mathcal{F}(u).\label{LuTime}
\end{equation}
The parameters $\mu=\mu(x,u(x,t),t),\,\eta=\eta(x,u(x,t),t)$ are
the mass and damping parameters. The idea in the infinite dimensional
setting is to solve (\ref{Lu}) by solving (\ref{LuTime}) in such
a way that $u_{t},\, u_{tt}\rightarrow0$ when $t\rightarrow t_{1},\, t_{1}\leq\infty$,
i.e., $\lim_{t\rightarrow t_{1}}u(x,t)=v(x).$ In addition, the two
initial conditions $u(t_{0})$ and $u_{t}(t_{0})$ are applied.

Both (\ref{Lu}) and (\ref{LuTime}) need to be discretized to attain
a numerical solution. For simplicity, we exemplify by applying finite
differences but it is possible to use, e.g., finite elements, basis
sets or any other method of discretization. We define a grid $x_{1},x_{2},\ldots$
and approximate $v(x_{i})$ by $v_{i}$ and assume that the discretized
version of (\ref{Lu}) can be written as \begin{equation}
F_{i}(v_{1}\ldots,v_{n})=0,\, i=1,\ldots,n\label{eq:Fi}\end{equation}
where $F_{i}:\mathbb{R}^{n}\rightarrow\mathbb{R}$.

Turning to the dynamical system (\ref{LuTime}) it is discretized
such that $u_{i}(t)$ approximates $u(x_{i},t)$ and $\mu_{i}(t)=\mu(x_{i},u_{i}(t),t),\,\eta_{i}(t)=\eta(x_{i},u_{i}(t),t)$
for $i=1,\ldots,n$. Further, $\mathcal{F}(u)$ is discretized as
$\mathcal{F}(v)$ in (\ref{eq:Fi}) and we approximate (\ref{LuTime})
with the system of ordinary differential equations \begin{equation}
\mu_{i}\ddot{u}_{i}+\eta_{i}\dot{u}_{i}=F_{i}(u_{1},\ldots,u_{n}),\, i=1,\ldots,n.\label{DFPM}\end{equation}
with initial conditions $u_{i}(t_{0}),\,\dot{u}_{i}(t_{0})$. Our
idea in the discrete setting is to solve (\ref{eq:Fi}) by solving
(\ref{DFPM}) such that $\dot{u}_{i}(t),\,\ddot{u}_{i}(t)\rightarrow0$
when $t\rightarrow t_{1},\, t_{1}\leq\infty$, i.e., $\lim_{t\rightarrow t_{1}}u_{i}(t)=v_{i}.$
The overall approach for solving (\ref{Lu}) using (\ref{DFPM}) is
named the \textsl{Dynamical Functional Particle Method}, DFPM.

\subsection{Related work and topics}

In a recent mechanics oriented article the connection between classical
particle methods and differential equations were studied \citep{EdGuPe11}.
This work had a clear focus on the physical understanding and mechanical
properties. The present work, however, turns the focus towards the
mathematical aspects in the attempt to answer questions related to
convergence, rate of convergence and the underlying reasons why DFPM
is seen to be efficient for some mathematical problems. The idea of
studying dynamical particle systems certainly has its origin in basic
physics and astronomy. The assumption there is that all matter consists
of particles. Their interactions are known and they follow the equations
of motion. The basic idea DFPM, however, is that the {}``forces''
and {}``particles'' instead are viewed as mathematical abstract
objects rather than physical. Thus mathematical quasi-particles are
formed. Their interactions are determined by the functional equation
at hand. From the mechanical point of view the quasi particles in
(\ref{eq:Fi}) have masses $\mu_{i}$ and all follow a dissipated
motion governed by $\eta_{i}$. Such Hamiltonian systems have many
interesting and useful properties, see, e.g., \citep{Goldstein}. In
Hamiltonian systems it is well known that symplectic integration techniques
are especially attractive \citep{Calvo,Reich}.

The idea of solving a time dependent problem to get the stationary
solution has also previously been applied in mathematics. A simple
example is the solution of an elliptic PDE such as in the heat equation,
see Sincovec and Madsen \citep{SiMa75}. Indeed, steady state solutions
are often the objective when simulating time-dependent PDE systems.
Since the stationary state is seeked, the evolution of the system
is considered to take place in \emph{artificial time}. The concept
of artificial time is further discussed and analyzed in \citep{AsHuDo07}.

A general approach is that of continuation, see \citep{AlGe90} for
an introduction, where (\ref{Lu}) is embedded in a family of problems
depending on a parameter $s$, i.e., \begin{equation}
\mathcal{F}(u;s)=0\label{eq:Fus}\end{equation}
where $\mathcal{F}(u;0)=\mathcal{F}(u)=0.$ Thus, solving (\ref{eq:Fus})
for $s=0$ is equivalent to solving (\ref{Lu}) and it is assumed
that solving (\ref{eq:Fus}) for some $s,$ say $s=1$, is computationally
cheap. The solution to (\ref{eq:Fus}) is found by solving a sequence
of problems for values of $s$ decreasing from $1$ to $0$. A general
package for continuation methods with additional references may be
found in Watson et al. \citep{WaSoMeMoWa97}. Further, see Nocedal
and Wright \citep{NoWr99} for a discussion in the context of optimization
and Ascher, Mattheij and Russell \citep{AsMaRu95} for boundary value
ODEs. DFPM can in principle be viewed as a sub-method to the group
of continuation methods. However, as far as the authors know, the
concrete application of a second order system (Hamiltonian dynamics)
to solve equations and the corresponding analysis as presented here
is novel.

Other works where (\ref{LuTime}) appear are for example the damped
harmonic oscillator in classical mechanics, the damped wave equation,
\citep{PaSq05} and the heavy ball with friction \citep{Al00}. These
problem settings are specific examples of physical systems and not
developed to solve equations in general.

In \citep{Ch88,Ch08} iterative processes to solve, e.g., eigenvalue
problems are considered as (gradient driven) dynamical systems. So
called fictitious time is used in \citep{TsLiYe10} where, e.g., Dirichlet
boundary value problem of quasilinear elliptic equation is numerically
solved by using the concept of fictitious time integration method.
The inverse problem of recovering a distributed parameter model is
considered in \citep{AsHuDo07} using the first order ODE attained
from the necessary optimality conditions of the inverse problem.

First order systems, mainly in the form $u_{t}=\mathcal{F}(u)$, have
been used to solve different kinds of equations $\mathcal{F}(v)=0$,
both as a general approach and intended for specific mathematical
problems. It is of interest to briefly consider the difference between
the first order differential equation $u_{t}=\mathcal{F}(u)$ and
the second order approach, DFPM. Suppose for simplicity that a discretization
is made by finite differences leading to a system of equations $\dot{u}_{i}=F_{i}(u_{1},\ldots,u_{n},\, x_{i})$.
Consider an example where the functional $\mathcal{F}(u)$ contains
a derivative w.r.t. \emph{u} (\emph{A}) and other functions of \emph{u}
and \emph{x} (\emph{B}). Then $F_{i}(u_{1},\ldots,u_{n},\, x_{i})=A_{i}(u_{1},\ldots,u_{n},\, x_{i})/h^{p}+B_{i}(u_{i},x_{i})$.
Dimensional analysis then gives that
$$\left[\dot{u}_{i}\right]=\frac{\left[u\right]}{\left[t\right]}=\left[A_{i}(u_{1},\ldots,u_{n},\, x_{i})/h^{p}\right]=\frac{\left[A\right]}{h^{p}}=\frac{\left[u\right]}{h^{p}}.$$
Given a certain component $F_{i}$ we have that $x_{i}$ is not variable,
so $\left[A\right]=\left[u\right]$. We see that the dimension of
time is related to the discretization, i.e., $t=\mathcal{O}\left(h^{p}\right)$.
In a similar way it can be shown that for the second order differential
equation (DFPM) one instead have that $t=\mathcal{O}\left(h^{p/2}\right)$.
In a numerical integration, the dimension of \emph{t} must still be
valid. This means that also the maximum timestep (for which a numerical
algorithm is stable) is given by $\triangle t_{max}=\mathcal{O}\left(h^{p}\right)$
for the first order case, and $\triangle t_{max}=\mathcal{O}\left(h^{p/2}\right)$
for the second order equation. Consider for example the case of a
central difference formula, i.e., \emph{p}=2. For the first order
differential equation we see that $\triangle t_{max}$ will have to
be very small as finer meshes are selected. This will lead to a less
efficient computational complexity for the first order system. The
complexity of DFPM will be further discussed in section \ref{sec:Helium}.

As we shall see, the DFPM algorithm seems attractive due to several
reasons. The most interesting points that will be studied are related
to computational complexity, easiness of implementation, Hamiltonian
dynamics and its relation to the total evolution time, stability and
cheapness of symplectic integration, exponential convergence and the
existence of potential energy.

\subsection{The outline of the paper}

The outline of the paper is based on illustrating the versatility
of  DFPM and to analyze the convergence aspects of the dynamical system
(\ref{DFPM}). In order to introduce the reader to DFPM, the damped
harmonic oscillator is revisited. DFPM clearly has a close relationship
to this type of classical system. It is then important to remind the
reader of the close connection to Hamiltonian dynamics in particular
the existence of a potential function. In such a case where the functional
is conservative any extreme value of its potential function is a solution
to the original problem (\ref{eq:Fi}). Specifically if the potential
has a minimum the solution of (\ref{DFPM}) will converge asymptotically.
This is dealt with in Section \ref{sec:Asymptotic-convergence}. A
Lyapunov function is applied to show asymptotic convergence in Section
\ref{sub:Using-the-Potential}. In Section \ref{sub:Linear-Convergence-Analysis}
we analyze the linearization of  DFPM and give precise statements
for local asymptotic convergence valid close to the solution and for
linear problems such as systems of linear equations. The rate of convergence
is treated in Section \ref{sec:Convergence-rate} with four subsections
treating the general problem (\ref{DFPM}) when there is a Lyapunov
function, the linearized problem, the choice of damping, and examples,
respectively. In Section \ref{sub:General-results-on} the Lyapunov
function is used to state a general theorem that together with additional
assumptions on the Lyapunov function gives an exponential convergence
rate. This theorem is then specialized to the case when there exists
a potential. The linearized problem is analyzed in Section \ref{sub:Convergence-rate-for}
where we first treat the case with one scalar damping and then discuss
the possibility to choose an optimal general damping matrix. The conclusions
drawn from the choice of damping in the linear case are used in Section
\ref{sub:Discussion-about-optimal} to formulate a local strategy
for the choice of optimal damping. To demonstrate the efficiency of
DFPM we report in the end of the article several examples. The most
noteworthy is the efficiency for treating symmetric eigenvalue problems.
It is shown that DFPM is order of magnitudes faster than the standard
software ARPACK \citep{website:arpack}. Finally, in Section \ref{sec:Conclusions-and-Future}
we make some conclusions, discuss open problems as well as suggestions
for future works.

\section{Two illuminating examples}

\subsection{The damped harmonic oscillator}

Despite its triviality an important example related to DFPM is the
damped harmonic oscillator where $F(u)=-ku,\, k>0$ which we include
for later reference since it illustrates many properties of (\ref{DFPM}).
 In this case the equation at hand is given by\begin{equation}
\mu\ddot{u}+\eta\dot{u}=-ku.\label{eq:harmonic oscillator}\end{equation}
In DFPM, as well as here, we set the inital condition $\dot{u}\left(0\right)=0$.
The initial condition for $u$ may be set arbitrary. In mechanics
the parameters $\mu>0$, $\eta>0$ and $k>0$ correspond to particle
mass, damping constant and spring constant, respectively. The time-dependent
solution is given explicitly by \[
u=c_{1}e^{-\gamma_{1}t}+c_{2}e^{-\gamma_{2}t}\]
where \begin{equation}
\gamma_{1,2}=-\frac{1}{2}\frac{\eta}{\mu}\pm\sqrt{\frac{1}{4}\frac{\eta^{2}}{\mu^{2}}-\frac{k}{\mu}}\label{eq:HarmRoots}\end{equation}
and $c_{1},c_{2}$ are constants given by the initial conditions.
Although in mechanics it is clear that all parameters in (\ref{eq:harmonic oscillator})
are physically positive, it is worthwhile to make a comment why this
is so. Consider the case $\mu<0,\eta<0$. The roots are then real
with one positive and one negative root so $u\left(t\right)$ will
diverge. The situation for $\mu<0,\eta>0$ is similar with one positive
real root and no convergence. When $\mu>0,\eta<0$ the roots may be
complex but one root will always have a positive real part and the
solution will not converge. Thus, the only possible choice for convergence
into a stationary solution is to apply positive parameters.

There are three different regimes of the parameters that will effect
the convergence: the under critical damping,$\eta<2\sqrt{k\mu}$ which
shows an oscillatory convergence, the critical damping, \foreignlanguage{swedish}{$\eta=2\sqrt{k\mu}$}
giving exponential convergence, and over critical damping, $\eta>2\sqrt{k\mu}$
resulting in a slower exponential convergence. The critical damped
system is known to be the fastest way for the system to return to
its equilibrium (i.e., the stationary solution) \citep{website:critical}.
It will be illustrative to return to this example later when considering
various aspects of convergence and convergence rate of DFPM in Sections
\ref{sec:Asymptotic-convergence} and \ref{sec:Convergence-rate}.

\subsection{A symmetric eigenvalue problem}

Symmetric eigenvalue problems are of great importance in physics and
technology. It is also a relatively straight forward example to illustrate
how the DFPM algorithm is applied. Consider the eigenvalue problem\begin{equation}
A\mathbf{v}=\lambda\mathbf{v}\label{eq:EigenvalueProblem}\end{equation}
where $A$ is a symmetric matrix with normalization $\left\Vert \mathbf{v}\right\Vert =1$.
The DFPM equation (\ref{DFPM}) is not directly applicable since the
eigenvalue $\lambda$ is unknown. However, it is well known, see \citep{GoVa96},
that the stationary solutions of the Rayleigh quotient\[
\rho(\mathbf{v})=\mathbf{v}^{T}A\mathbf{v},\:\left\Vert \mathbf{v}\right\Vert =1\]
are eigenvectors of (\ref{eq:EigenvalueProblem}). Specifically, we
have that \[
\arg\min\rho(\mathbf{v})=\lambda_{min}\]
where $\lambda_{min}$ is the smallest eigenvalue of $A$. Thus, one
way to formulate the functional vector $\mathbf{F}$ is

\[
\mathbf{F}=-A\mathbf{u}+\left(\mathbf{u}^{T}A\mathbf{u}\right)\mathbf{u},\:\left\Vert \mathbf{u}\right\Vert =1\]
where $\mathbf{u}=\mathbf{u}\left(t\right)$. The DFPM equation (\ref{DFPM})
is then given by \begin{equation}
M\ddot{\mathbf{u}}+N\dot{\mathbf{u}}=-A\mathbf{u}+\left(\mathbf{u}^{T}A\mathbf{u}\right)\mathbf{u},\:\left\Vert \mathbf{u}\right\Vert =1\label{eq:eigDFPM}\end{equation}
where $M=\textrm{diag}(\mu_{1},\ldots,\mu_{n}),\, N=\textrm{diag}(\eta_{1},\ldots,\eta_{n})$.
This procedure will yield $\lambda_{min}$ and its corresponding eigenvector
$\mathbf{u}$. We shall see later that by replacing $\mathbf{F}$
with $-\mathbf{F}$ we instead get $\lambda_{max}$ and its corresponding
eigenvector. There are various strategies to get the other solutions.
One possibility is to apply a Gram-Schmidt process \citep{bjor:94}.
Often in applications, only a few of the lowest solutions are of interest.
The reader should note that in practice the matrix $A$ never needs
to be formulated explicitly. In DFPM one instead works with the components
of the functional vector $F_{i}$. The mechanical interpretation is
of course that this is the force acting on particle \emph{i}. The
formulation therefore automatically becomes sparse. This is later
illustrated in Section \ref{sec:Helium}.

\section{The Potential and Total Energy\label{sec:Potential Energy}}

DFPM can be considered as a many-particle system in classical mechanics
where $F_{i}$ is the force acting on particle \emph{i}, $\mu_{i}$
is its point mass, and $-\eta_{i}\dot{u}_{i}$ is the damping force
\citep{Goldstein}. In this section we revisit the concept of a conservative
force field and thus the existence of a many-particle potential. In
DFPM the functional is not necessarily conservative, but if it is,
the analysis is greatly simplified. By using the results in this section,
we shall see in Section \ref{sub:Using-the-Potential} that for a
convex potential, the stationary solution to (\ref{DFPM}) corresponds
to a minimum of the many-particle potential.

We start by taking the view that (\ref{eq:Fi}) is a vector field
in $\mathbb{R}^{n}$: \begin{equation}
\mathbf{F}=(F_{1},F_{2},\ldots,F_{n}),\, F_{i}=F_{i}(u_{1},u_{2},\ldots,u_{n}).\label{eq:vecField}\end{equation}

\begin{defn}
\label{def:Conservative}The vector field $\mathbf{F}$ in (\ref{eq:vecField})
is \textit{conservative} if there exists a potential function $V:\mathbb{R}^{n}\rightarrow\mathbb{R}$
such that $\mathbf{F}=-\nabla V$.
\end{defn}
For any conservative field $\mathbf{F}$ we have that \[
\mathbf{F}(\mathbf{v})=0\Leftrightarrow\nabla V(\mathbf{v})=0.\]
Thus, any solution to (\ref{eq:Fi}) is an extreme value of the potential
$V$, i.e., a minimum, maximum or saddle point. In other words, solving
(\ref{eq:Fi}) by  DFPM is equivalent to finding the extreme points
of the potential $V$. We will explore this fact further when analyzing
the convergence of DFPM in Section \ref{sec:Asymptotic-convergence}.

By differentiating $\nabla V$ and assuming that $V$ is at least
twice continuously differentiable, we get

\begin{equation}
\frac{\partial F_{i}}{\partial u_{j}}-\frac{\partial F_{j}}{\partial u_{i}}=0,1\leq i,j\leq n.\label{eq:FiFj}\end{equation}
as a necessary and sufficient condition for $\mathbf{F}$ to be conservative.
Note that one good example fulfilling the condition (\ref{eq:FiFj})
is the force field $\mathbf{F}(\mathbf{u})=A\mathbf{u}$ where $A$
is a symmetric matrix. We shall see later that this fact is very useful
for symmetric eigenproblems.

It is possible to derive the condition (\ref{eq:FiFj}) that is interesting
in its own since it contains the possibility to consider equations
on manifolds. Consider the (work) $1$-form $\varphi=\sum F_{j}du_{j}$,
see \citep{HuHu09} for a definition of $k$-forms. The $1$-form $\varphi$
is said to be closed if $d\varphi=0$ and Poincarés Lemma \citep{Co93}
implies that any closed form is exact, i.e., in our context has a
potential, say $V,$ that is $dV=\varphi$. There is no ambiguity
to say that $V$ is a potential as in Definition \ref{def:Conservative}.
We have the following results for the vector field in (\ref{eq:vecField}).
\begin{thm}
The following statements are equivalent:

1. $\mathbf{F}$ is conservative, that is $\mathbf{F}=-\nabla V$

2. $dV=\varphi$

3. $\int_{\Gamma}\mathbf{F\cdot\mathrm{d\mathbf{r}}}$ is independent
of the path $\Gamma$

4. $\oint_{\Gamma}\mathbf{F}\cdot\mathrm{d}\mathbf{r}=0$ for all
closed paths $\Gamma$.\end{thm}
\begin{proof}
The equivalence of 1 and 2 is trivial. The equivalence of 1 and 3
can be found in \citep{HuHu09}. Now, assume that 3. is true. Let $p$
and $q$ be two arbitrary points, and let $\gamma_{1}$ and $\gamma_{2}$
be two piecewise smooth paths from $p$ to $q$. Define $\Gamma$
as the closed path which first goes from $p$ to $q$, via $\gamma_{1}$,
and then from $q$ to $p$ via $-\gamma_{2}$. Then, since $\int_{\gamma_{1}}\mathbf{F\cdot\mathrm{d\mathbf{r=\int_{\gamma_{2}}\mathbf{F}\mathrm{\cdot d\mathbf{r}}}}}$,
we get that

\[
0=\int_{\gamma_{1}}\mathbf{F\cdot\mathrm{d\mathbf{r}-}\int_{\gamma_{2}}\mathbf{F}\cdot\mathrm{d\mathbf{r}=}\int_{\Gamma}\mathbf{F}\cdot\mathrm{d\mathbf{r}}}.\]

Since $p$ and $q$ are arbitrary points, and $\gamma_{1}$ and $\gamma_{2}$
are arbitrary, the closed path $\Gamma$ is arbitrary, and therefore
the implication 3 to 4 is proved. The implication 4 to 3 is similar.
\end{proof}
If a potential exists it can be derived from the discretized equations
by calculating the work, $W$, simply integrating along any path,
say, from $\mathbf{0}$ to $\mathbf{u}$. For example it is possible
to use coordinate directions as\begin{eqnarray}
W & = & \intop_{0}^{u_{1}}F_{1}(s_{1,}0,\ldots,0)ds_{1}+\intop_{0}^{u_{2}}F_{2}(u_{1,}s_{2},0,\ldots,0)ds_{2}+\ldots\label{eq:W}\\
 &  & +\intop_{0}^{u_{n}}F_{n}(u_{1},u_{2},\ldots,u_{n-1},s_{n})ds_{n}=-V.\nonumber \end{eqnarray}

\subsubsection{A revisit to the symmetric eigenvalue problem\label{sub:The-Quantum-Harmonic}}

Recall that DFPM equation for the symmetric eigenvalue problem (\ref{eq:eigDFPM}).
The corresponding vector field \[
\mathbf{F}=-A\mathbf{u}+\left(\mathbf{u}^{T}A\mathbf{u}\right)\mathbf{u},\:\left\Vert \mathbf{u}\right\Vert =1\]
is conservative with the potential \begin{equation}
V(\mathbf{u})=\frac{1}{2}\mathbf{u}^{T}A\mathbf{u},\:\left\Vert \mathbf{u}\right\Vert =1\label{eq:EigPot}\end{equation}
To prove this it would at a first glance seem natural to find the
gradient of $V(\mathbf{u})$. However, the normalization $\left\Vert \mathbf{u}\right\Vert =1$
complicates this somewhat. This can be treated by investigate the
gradient on the sphere $S^{n-1}=\left\{ \mathbf{u}\in\mathbb{R}^{n}:,\:\left\Vert \mathbf{u}\right\Vert =1\right\} $.
Denote the tangent space to the sphere at a point $\mathbf{\mathbf{u}}$
as $T_{\mathbf{\mathbf{u}}}(S^{n-1})$. By using the Euclidean metric
(the 2-norm), the gradient of $V$ at $\mathbf{u}$ is the unique
vector $\nabla_{S^{n-1}}V(\mathbf{u})\in T_{\mathbf{\mathbf{u}}}(S^{n-1})$
such that\[
\nabla V(\mathbf{u})^{T}\mathbf{t}=\nabla_{S^{n-1}}V(\mathbf{u})^{T}\mathbf{t}\]
for all tangent vectors \foreignlanguage{swedish}{$\mathbf{t}\in T_{\mathbf{\mathbf{u}}}(S^{n-1})$}
where $\nabla V(\mathbf{u})$ is the usual gradient in $\mathbb{R}^{n}$.
Solving this equation for $\nabla_{S^{n-1}}V(\mathbf{u})$ by realizing
that $\nabla_{S^{n-1}}V(\mathbf{u})$ is the projection of $\nabla V$
on $T_{\mathbf{\mathbf{u}}}(S^{n-1})$ we get\[
\nabla_{S^{n-1}}V(\mathbf{u})=(I-\mathbf{\mathbf{n}(\mathbf{u})}\mathbf{n}(\mathbf{u})^{T})\nabla V(\mathbf{u})=\nabla V(\mathbf{u})-\left(\mathbf{n}(\mathbf{u})^{T}\nabla V(\mathbf{u})\right)\mathbf{n}(\mathbf{u})\]
where $\mathbf{n}(\mathbf{u})$ is the normal to $T_{\mathbf{u}}(S^{n-1})$.
Since, for $S^{n-1}$ we have $\mathbf{n}(\mathbf{u})=\mathbf{u}$
and we get \[
\nabla_{S^{n-1}}V(\mathbf{u})=A\mathbf{u}-\left(\mathbf{u}^{T}A\mathbf{u}\right)\mathbf{u}=-\mathbf{F}(\mathbf{u})\]
showing that $V$ in fact is a potential to the vector field $\mathbf{F}$.

\section{Asymptotic convergence\label{sec:Asymptotic-convergence}}

In this section we investigate the convergence properties of the solution
$\mathbf{u}(t)$ given by (\ref{DFPM}). Since we are interested in
the asymptotic solution we will use stability theory for dynamical
systems, see \citep{MiHoLi08}, namely the use of a Lyapunov function
and local linear analysis. However, it is generally difficult to find
a Lyapunov function. Consequently, we start with the case where there
exists a potential and where the Lyapunov function can be chosen as
the total energy. If no potential exists we are left with the linear
stability analysis in Section \ref{sub:Linear-Convergence-Analysis}.

For simplicity and without loss of generality we may assume that $\mu_{i}\equiv1$
and that the solution of (\ref{eq:Fi}) is $\hat{\mathbf{u}}=0$.

\subsection{Using the Potential\label{sub:Using-the-Potential}}

In this section we assume that there exists a potential, $V(\mathbf{u})$,
to the given equations in (\ref{eq:Fi}). Then (\ref{DFPM}) may be
written as\begin{equation}
\ddot{\mathbf{u}}+N\dot{\mathbf{u}}=-\nabla V\label{eq:gradDFPM}\end{equation}
where $N=\textrm{diag}(\eta_{1},\ldots,\eta_{n})$. The energy functional
(the Lyapunov function) is given by
\begin{equation}
E=T+V\label{eq:E}
\end{equation}
where
$$T=\frac{1}{2}\sum\dot{u}_{i}^{2}$$
is the kinetic energy. We then have the following important result
to be used in the analysis of the asymptotic convergence analysis.
\begin{lem}
\label{lem:dEdt} Assume that $\eta_{i}>0,\, i=1,\ldots,n$. For the
given solution of the dynamical system (\ref{DFPM}) the energy functional
defined in (\ref{eq:E}) is a non-increasing function.\end{lem}
\begin{proof}
From the definition of $E$ in (\ref{eq:E}) and $\nabla V=-\mathbf{F}$
we have
$$ \frac{dE}{dt}=\frac{dT}{dt}+\frac{dV}{dt}=\sum\dot{u}_{i}\ddot{u}_{i}+\frac{\partial V}{\partial u_{i}}\dot{u}_{i}=\sum\dot{u}_{i}\ddot{u}_{i}-F_{i}\dot{u}_{i}$$
and since $F_{i}=\ddot{u}_{i}+\eta_{i}\dot{u_{i}}$ we get
\begin{equation}
\frac{dE}{dt}=-\sum_{i}\eta_{i}\dot{u}_{i}^{2}\leq0.\label{eq:dEdt}
\end{equation}
\end{proof}

Lemma \ref{lem:dEdt} tells us that the energy is non-increasing which
is not surprising from a mechanical point of view since the damping
will decrease the total amount of energy and there are no additional
sources of energy in the system.

The next theorem is taken from \citep{MiHoLi08}. The proof is omitted.
\begin{thm}
\label{thm:MHLTheorem}Consider the autonomous system\begin{equation}
\dot{\mathbf{w}}=G(\mathbf{w})\label{eq:wpGw}\end{equation}
 where $G:\Omega\to\mathbb{R}^{n}$ is a continuous function defined
on a domain $\Omega$ in $\mathbb{R}^{n}$containing the origin and
$G(0)=0$. Assume that the Lyapunov function $L$ is non-negative
with respect to (\ref{eq:wpGw}) for all $\mathbf{w}\in\Omega$ and
such that for some constant $c\in\mathbb{R}$ the set $H_{c}$ is
a closed and bounded component of the set $\{\mathbf{w}\in\Omega:L(\mathbf{w})\leq c\}$.
Let $M$ be the largest invariant set in the set\[
Z=\left\{ \mathbf{w}\in\Omega:\frac{dL}{dt}(\mathbf{w})=0\right\} .\]
Then every solution $\mathbf{w}(t)$ of the autonomous system (\ref{eq:wpGw})
with $\mathbf{w}(t_{0})\in H_{c}$ approaches the set $M$ as $t\to\infty$.
\end{thm}
Using Theorem \ref{thm:MHLTheorem} we are now able to show our main
result in this section for the asymptotic convergence of (\ref{eq:gradDFPM}).
\begin{thm}
\label{thm:ConvDFPMPotential}Assume that there exists a solution,
$\hat{\mathbf{u}}$, to (\ref{eq:Fi}) and a potential $V$ to $\mathbf{F}$.
If $V$ is globally convex, i.e., convex in the whole of $\mathbb{R}^{n}$,
then for any initial starting point in $\mathbb{R}^{n}$ the solution
of (\ref{eq:gradDFPM}) will be asymptotically convergent to $\hat{\mathbf{u}}$.
If $V$ is locally convex in a neighbourhood $\Upsilon$ of $\hat{\mathbf{u}}$
then for any initial starting point in $\Upsilon$ the solution of
(\ref{eq:gradDFPM}) will be asymptotically convergent to $\hat{\mathbf{u}}$. \end{thm}
\begin{proof}
Rewrite  DFPM (\ref{eq:gradDFPM}) as a system by letting $\mathbf{v}=\mathbf{\dot{u}}$
as \begin{equation}
\left[\begin{array}{c}
\dot{\mathbf{u}}\\
\dot{\mathbf{v}}\end{array}\right]=\left[\begin{array}{c}
\mathbf{v}\\
-N\mathbf{v}-\nabla_{\mathbf{u}}V\end{array}\right]\label{eq:DynSyst}\end{equation}
We want to use Theorem \ref{thm:MHLTheorem} to prove asymptotic convergence
of DFPM. From Lemma \ref{lem:dEdt} we have that $dE/dt=-\mathbf{v}^{T}N\mathbf{v}\leq0$
and therefore $dE/dt=0$ if and only if $\mathbf{v}=0.$ Define\[
\mathbf{w}=\left[\begin{array}{c}
\mathbf{u}\\
\mathbf{v}\end{array}\right].\]
 Let $M=Z=\{\mathbf{u}:dE/dt=0\}=\{\mathbf{w}:\mathbf{v}=0\}$ be
the invariant set in Theorem \ref{thm:MHLTheorem}. Then, by Theorem
\ref{thm:MHLTheorem} again, the solution $\mathbf{w}$ to the system
(\ref{eq:DynSyst}) approaches the set $M$ as $t\to\infty.$ If $\mathbf{w}\in M$
then $\mathbf{v}=\dot{\mathbf{u}}=0$, so from (\ref{eq:DynSyst})
we have that $\dot{\mathbf{v}}=-\nabla_{\mathbf{u}}V\neq0$ if $\mathbf{u}\neq0=\hat{\mathbf{u}}$
and $\mathbf{w}$ can not remain in the set $M$ if $\mathbf{u}\neq\hat{\mathbf{u}}$.
We need to verify that $\mathbf{u}=0$ as $t\to\infty,$ but this
follows from the fact that $E$ is non-increasing. Hence $\mathbf{w}\to0$
as $t\to\infty.$
\end{proof}

\subsection{Linear Convergence Analysis\label{sub:Linear-Convergence-Analysis}}

Without a Lyapunov function and with no existing potential one is
left with a local linear stability analysis. Such an analysis is based
on the linearization of the dynamical system (\ref{DFPM}) at a solution
, $\hat{\mathbf{u}}=0$, to (\ref{eq:Fi}). Define $J(\mathbf{u})$
as the Jacobian of $\mathbf{F}$. From the Taylor expansion $\mathbf{F}(\mathbf{u})=\mathbf{F}(0)+J(0)\mathbf{u}+\mathcal{O}(\left\Vert \mathbf{u}^{2}\right\Vert )$
we define the linearized problem to (\ref{DFPM}) as
 \begin{equation}
\hat{M}\ddot{\mathbf{u}}+\hat{N}\dot{\mathbf{u}}=\mathbf{F}(0)+J(0)\mathbf{u}\label{eq:linFJu}
\end{equation}
where
$$
\hat{M}\textrm{=diag}\mathit{\mathrm{(\hat{\mu}_{1},\ldots,\hat{\mu}_{m})}},\,
\hat{N}\textrm{=diag}\mathit{\mathrm{(\hat{\eta}_{1},\ldots,\hat{\eta}_{m})}}$$
and $\hat{\mu}_{i},\hat{\eta}_{i}$ are the values of $\mu,\eta$
at $\hat{\mathbf{u}}$. Thus, the local convergence can be analyzed
by analyzing the linear system (\ref{eq:linFJu}) which for notational
convenience is written
\begin{equation}
M\ddot{\mathbf{u}}+N\dot{\mathbf{u}}+A\mathbf{u}=b\label{LinDyn}\end{equation}
where $M,N,A\in\mathfrak{\mathbb{R}}^{n\times n},\mathbf{b}\in\mathbb{R}^{n}$.
In \citep{DiYe99} sufficient conditions are given for (\ref{LinDyn})
to have asymptotically stable solutions. In order to state these conditions
we need some additional notation. Consider the homogeneous equation,
i.e., \begin{equation}
M\ddot{\mathbf{u}}+N\dot{\mathbf{u}}+A\mathbf{u}=0.\label{HomLin}\end{equation}
By inserting the eigensolution $e^{\xi_{i}t}\mathbf{v}_{i}$ into
(\ref{HomLin}) we get the equation \begin{equation}
(\xi_{i}^{2}M+\xi_{i}N+A)\mathbf{v}_{i}=0\label{charpol}\end{equation}
for the eigenvalue $\xi_{i}$ and eigenvector $\mathbf{v}_{i}$ of
$A$. Let $\mathrm{Re}(\mathbf{v}_{i})$ , $\mathrm{Im}(\mathbf{v}_{i})$
denote the real and imaginary part of $\mathbf{v}_{i}$, respectively.
Introduce the symmetric and anti-symmetric parts of $M,N,A$ as $M_{S},M_{A},N_{S},...$.
and define \[
s_{i}(M)=\mathrm{Re}(\mathbf{v}_{i})^{T}M_{S}\mathrm{Re}(\mathbf{v}_{i})+\mathrm{Im}(\mathbf{v}_{i})^{T}M_{S}\mathrm{Im}(\mathbf{v}_{i}),\, a_{i}(M)=2\mathrm{Re}(\mathbf{v}_{i})^{T}M_{A}\mathrm{Im}(\mathbf{v}_{i})\]
with corresponding definitions for $s_{i}(N),a_{i}(N),s_{i}(A),a_{i}(A)$.
By applying the general Hurwitz criterion, see e.g. \citep{Ga59},
to (\ref{charpol}) we get the following theorem and corollary. For
a detailed presentation of the proofs we refer to \citep{DiYe99}.
\begin{thm}
\textsl{\label{thm:AsymLinear}The solution to (\ref{LinDyn}) will
converge asymptotically if and only if
$$ s_{i}(M)s_{i}(N)+a_{i}(M)a_{i}(N)>0$$
and}
\begin{align*}
\left(s_{i}(N)s_{i}(A)+a_{i}(N)a_{i}(A)\right)\left(s_{i}(M)s_{i}(N)+a_{i}(M)a_{i}(N)\right) & -\\
-\left(s_{i}(M)a_{i}(A)-a_{i}(M)s_{i}(A)\right)^{2} & >0\end{align*}
\end{thm}
\begin{cor}
\textsl{\label{cor:LinConv}If $M,N$ are positive definite and $A$
has eigenvalues with positive real parts then the solution to (\ref{LinDyn})
will converge asymptotically.}
\end{cor}
Let us now return to the question of local convergence for (\ref{DFPM})
and state the following theorem.
\begin{thm}
\label{thm:LocalConv}Define $J(\mathbf{u})$ as the Jacobian of $\mathbf{F}$.
Assume that there exists a solution, $\hat{\mathbf{u}}$, to (\ref{eq:Fi}).
Further, assume that \[
\hat{M}\textrm{=diag}\mathit{\mathrm{(\hat{\mu}_{1},\ldots,\hat{\mu}_{m})}},\,\hat{N}\textrm{=diag}\mathit{\mathrm{(\hat{\eta}_{1},\ldots,\hat{\eta}_{m})}}\]
where $\hat{\mu}_{i},\hat{\eta}_{i}$ are the values of $\mu,\eta$
at $\hat{\mathbf{u}}$. Then DFPM will \textsl{converge asymptotically
}for any initial starting point of (\ref{DFPM}) close enough to $\hat{\mathbf{u}}$\textsl{
if and only if} the conditions in Theorem \ref{thm:AsymLinear} are
fulfilled where $M=\hat{M},N=\hat{N},A=-J(\hat{\mathbf{u}})$. Further,
if $\hat{M},\hat{N}$ are positive definite and $J(\hat{\mathbf{u}})$
has eigenvalues with negative real parts then DFPM will \textsl{converge
asymptotically }for any initial starting point of (\ref{DFPM}) close
enough to $\hat{\mathbf{u}}$\textsl{.}\end{thm}
\begin{proof}
The first statement in the theorem follows directly from Theorem \ref{thm:AsymLinear}
and the second from Corollary \ref{cor:LinConv}.
\end{proof}

\section{Convergence rate\label{sec:Convergence-rate}}

\subsection{General results on convergence rate\label{sub:General-results-on}}

Sharp estimates of the convergence rate for DFPM in a general case
is difficult and not realistic. However, we shall give some important
special cases that is relevant for solving equations with DFPM, i.e.,
to achieve fast \textit{exponential} convergence. We emphasize that
this is crucial to attain an efficient method for solving (\ref{eq:Fi}).
We again assume without loss of generality that the solution of (\ref{eq:Fi})
is $\hat{\mathbf{u}}=0$.
\begin{defn}
The solution, $\hat{\mathbf{u}}$, to (\ref{eq:Fi}) is \textit{locally}
exponentially stable and the solution $\mathbf{u}(t)$ to (\ref{DFPM})
has a \textit{local} exponential convergence rate if there exists
an $\alpha>0$ and for every $\varepsilon>0$ and every $t_{0}\ge0$,
there exists a $\delta\left(\varepsilon\right)>0$ such that for all
solutions of (\ref{DFPM}) $\left\Vert \mathbf{u}(t)\right\Vert \leq\varepsilon e^{-\alpha(t\text{\textminus}t_{0})}$
for all $t\ge t_{0}$ whenever $\left\Vert \mathbf{u}(t_{0})\right\Vert <\delta\left(\varepsilon\right)$.
The solution, $\hat{\mathbf{u}}$, to (\ref{eq:Fi}) is \textit{globally}
exponentially stable and the solution $\mathbf{u}(t)$ to (\ref{DFPM})
has a \textit{global} exponential convergence rate if for $\beta>0$
there exists $k(\beta)>0$ such that $\left\Vert \mathbf{u}(t)\right\Vert \leq k(\beta)e^{-\alpha(t-t_{0})}$
whenever $\left\Vert \mathbf{u}(t_{0})\right\Vert <k(\beta)$.
\end{defn}
We begin by stating a general theorem from \citep{MiHoLi08} giving
one possible formulation of the requirements for exponential convergence
based on the existence of a Lyapunov function $L(\mathbf{u},\mathbf{\dot{u}})$.
\begin{thm}
\label{thm:ExpLyapunov}Assume that there exist a Lyapunov function
$L=L(\mathbf{u},\mathbf{\dot{u}}):B(r)\rightarrow\mathbb{R}$ , $B(r)=\left\{ \left\Vert \left[\begin{array}{c}
\mathbf{u}\\
\mathbf{\dot{u}}\end{array}\right]\right\Vert \leq r\right\} $ and four positive constants $c_{1},c_{2},c_{3}$, $p$ such that\[
c_{1}\left\Vert \left[\begin{array}{c}
\mathbf{u}\\
\mathbf{\dot{u}}\end{array}\right]\right\Vert ^{p}\text{\ensuremath{\le}}L(\mathbf{u},\mathbf{\dot{u}})\text{\ensuremath{\le}}c_{2}\left\Vert \left[\begin{array}{c}
\mathbf{u}\\
\mathbf{\dot{u}}\end{array}\right]\right\Vert ^{p}\]
and\[
\frac{dL}{dt}\leq-c_{3}\left\Vert \left[\begin{array}{c}
\mathbf{u}\\
\mathbf{\dot{u}}\end{array}\right]\right\Vert ^{p}\]
for all $\left[\begin{array}{c}
\mathbf{u}\\
\mathbf{\dot{u}}\end{array}\right]\in B(r)$. Then the solution, $\hat{\mathbf{u}}$, to (\ref{eq:Fi}) is locally
exponentially stable and the solution $\mathbf{u}(t)$ to (\ref{DFPM})
has a local exponential convergence rate. If $B(r)$ is the whole
$\mathbb{R}^{2n}$ the solution $\hat{\mathbf{u}}$ to (\ref{eq:Fi})
is globally exponentially stable and the solution $\mathbf{u}(t)$
to (\ref{DFPM}) has a global exponential convergence rate.
\end{thm}
In the case that there exists a potential we can choose the Lyapunov
function as the total energy $E$ in (\ref{eq:E}). We will state
a theorem that proves exponential convergence in this case that is
slightly different from Theorem \ref{thm:ExpLyapunov}.
\begin{thm}
\label{thm:ExpTotalEnergy}Assume that there exists a potential $V(\mathbf{u})$
which is locally convex in a neighbourhood of the solution, $\hat{\mathbf{u}}$,
to (\ref{eq:Fi}) and satisfies the bound\begin{equation}
c_{1}\left\Vert \mathbf{u}\right\Vert ^{p}\text{\ensuremath{\le}}V(\mathbf{u})\text{\ensuremath{\le}}c_{2}\left\Vert \mathbf{u}\right\Vert ^{p}\label{eq:Vbound}\end{equation}
where $c_{1},c_{2}$ and $p$ are positive constants. Then $\hat{\mathbf{u}}$
is locally exponentially stable and the solution $\mathbf{u}(t)$
to (\ref{DFPM}) has a local exponential convergence rate. If the
potential is convex and satisfies the bound (\ref{eq:Vbound}) in
the whole $\mathbb{R}^{n}$ the solution $\mathbf{u}(t)$ to (\ref{DFPM})
has a global exponential convergence rate.\end{thm}
\begin{proof}
From (\ref{eq:Vbound}) and the definition of kinetic energy we have
that the total energy is bounded as\begin{equation}
c_{1}\left\Vert \mathbf{u}\right\Vert ^{p}+\frac{1}{2}\mu_{max}\left\Vert \mathbf{\dot{u}}\right\Vert ^{2}\leq E\leq c_{2}\left\Vert \mathbf{u}\right\Vert ^{p}+\frac{1}{2}\mu_{max}\left\Vert \dot{\mathbf{u}}\right\Vert ^{2}\label{eq:Ebound}\end{equation}
Thus, from Lemma \ref{lem:dEdt} we have \[
\frac{\frac{dE}{dt}}{E}\leq\frac{-\eta_{max}\left\Vert \dot{\mathbf{u}}\right\Vert ^{2}}{c_{2}\left\Vert \mathbf{u}\right\Vert ^{p}+\frac{1}{2}\mu_{max}\left\Vert \dot{\mathbf{u}}\right\Vert ^{2}}\]
If $V$ is convex we see from Theorem \ref{thm:ConvDFPMPotential}
that $\left\Vert \dot{\mathbf{u}}\right\Vert >0$ unless at the solution
$\hat{\mathbf{u}}$ and therefore there exists a constant $\gamma>0$
such that \[
\frac{dE}{dt}\leq-\gamma E\]
and thus $E\leq E_{0}e^{-\gamma(t-t_{0})}$. From (\ref{eq:Ebound})
we see that \[
c_{1}\left\Vert \mathbf{u}\right\Vert ^{p}+\frac{1}{2}\mu_{max}\left\Vert \mathbf{\dot{u}}\right\Vert ^{2}\leq E\leq E_{0}e^{-\gamma(t-t_{0})}\]
and this implies\[
\left\Vert \mathbf{u}\right\Vert \leq ce^{-\frac{1}{p}\gamma(t-t_{0})}\]
for some positive constant $c$ depending only on the initial conditions.
\end{proof}

\subsection{Convergence rate for linear problems\label{sub:Convergence-rate-for}}

Consider again the damped harmonic oscillator (\ref{eq:harmonic oscillator}).
Obviously, the convergence rate is exponential for all different dampings.
However, the fastest convergence rate is achieved for the case where
the damping is chosen as critical damping which can be seen in (\ref{eq:HarmRoots}):
The negative real part below critical damping is $\eta/2$ and therefore
$\eta$ should be as large as possible. However, as soon as critical
damping is exceeded, one of the roots will be real and the negative
real part will increase for the larger real root.

This property is inherited for more general linear problems defined
by (\ref{LinDyn}) which we now shall investigate further. We will
assume that $M$ and $N$ are diagonal matrices with positive elements
and then we can, without loss of generality, assume that $M=I,\mathbf{b}=0$
and consider the system\begin{equation}
\ddot{\mathbf{u}}+N\dot{\mathbf{u}}+A\mathbf{u}=0,N\textrm{=diag}\mathit{\mathrm{(\eta_{1},\ldots,\eta_{n})}}\label{eq:LinNA}\end{equation}
where $A\in\mathbb{R}^{n\times n}$. This linear system can be restated
as \[
\left[\begin{array}{c}
\dot{\mathbf{u}}\\
\dot{\mathbf{v}}\end{array}\right]=\left[\begin{array}{rr}
0 & I\\
-A & -N\end{array}\right]\left[\begin{array}{c}
\mathbf{u}\\
\mathbf{v}\end{array}\right]\]
and since this is a linear autonomous system of first order differential
equations, it is well known that any convergent solution will have
exponential convergence rate. However, we are interested in solving
the linear equations as fast as possible and then it is reasonable
to try to answer the question: How fast is the exponential rate of
convergence? This is a difficult question for a general $A$ and $N$
because of the following result from linear system theory \citep{Ph03}.
\begin{thm}
\label{thm:Similarity}Any two square matrices that are diagonalizable
have the same similarity transformation if and only if they commute.
\end{thm}
For the problem (\ref{eq:LinNA}), Theorem \ref{thm:Similarity} means
that $AN=NA$ is a necessary and sufficient condition for having a
similarity transformation such that $A=T\Lambda_{A}T^{-1},N=T\Lambda_{N}T^{-1}$
where $\Lambda_{A},\Lambda_{N}$ are diagonal matrices. In other words,
the two matrices $A$ and $N$ has to commute in order to decouple
the system (\ref{eq:LinNA}) into $n$ one dimensional damped harmonic
oscillators where the optimal damping is critical damping as shown
earlier. Note that the special case $N=\eta I$ where all damping
parameters are the same, trivially commutes with any matrix $A$.

\subsection{A discussion of optimal damping\label{sub:Discussion-about-optimal}}

The problem of choosing the damping in (\ref{DFPM}) such that the
asymptotic solution is attained, to some precision, in minimal time
is difficult, see, e.g., \citep{CaCo01},\citep{ShLaTo92} for some
special cases. Moreover, from a practical point of view this is not
very interesting since it requires \textit{a priori} knowledge of
the solution. A more interesting approach is to choose the damping
according to a local measure of curvature which we will discuss briefly.
Assume that the solution to (\ref{eq:Fi}) is $\hat{\mathbf{u}}=0$
and that there exists a potential $V$ that is convex with $V(\hat{\mathbf{u}})=0$
and a positive definite Hessian $\nabla^{2}V(u)$. Consider the case
with a single damping parameter $\eta=\eta(t)$. Then a Taylor expansion
at $\hat{\mathbf{u}}$ in (\ref{DFPM}) gives the approximate problem\begin{equation}
\ddot{\mathbf{u}}+\eta(t)\dot{\mathbf{u}}=-\nabla^{2}V(0)\mathbf{u}\label{eq:TaylorDFPM}\end{equation}
From the linear case treated in Section \ref{sub:Linear-Convergence-Analysis}
the optimal damping for (\ref{eq:TaylorDFPM}) is $\eta\thickapprox2\sqrt{\lambda_{min}(\nabla^{2}V(0))}$
where $\lambda_{min}(\cdot)$ denotes the smallest positive eigenvalue.
Now, consider any $\mathbf{u}(t),t\geq t_{0}$ then we conjecture
that a good choice of damping is $\eta(t)=2\sqrt{\lambda_{min}(\nabla^{2}V(\mathbf{u}(t)))}$.
To illustrate the possibilities of this choice of damping we given
an example with a potential \begin{equation}
V=e^{u_{1}^{2}+2u_{2}^{2}}\label{eq:expPot}\end{equation}
that is globally convex with a minimum at $u_{1}=u_{2}=0$. %
\begin{figure}
\begin{centering}
\includegraphics[bb=0bp 0bp 1024bp 483bp,width=8.6cm]{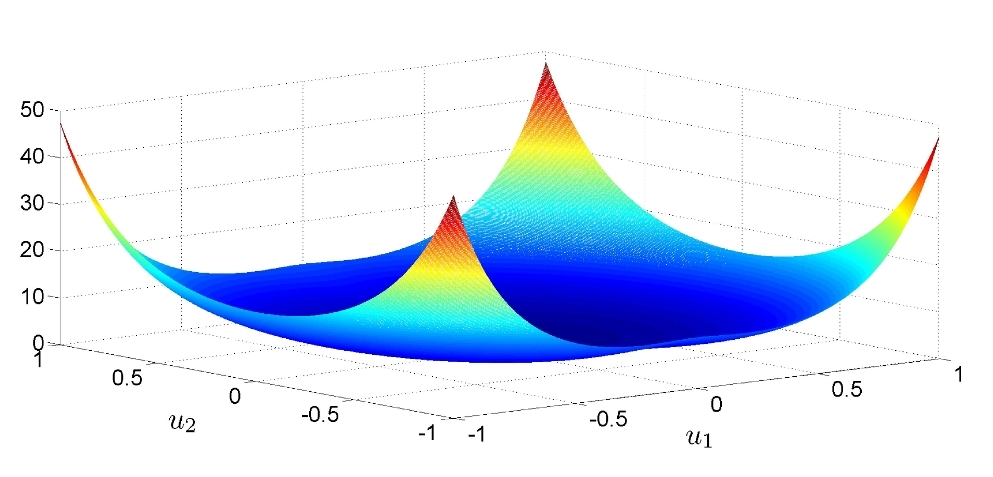}
\par\end{centering}

\caption{Smallest eigenvalue of the Hessian $\nabla^{2}V(\mathbf{u})$ for
the potential (\ref{eq:expPot}).\label{fig:EigDistribution}}

\end{figure}
In Figure \ref{fig:EigDistribution} the eigenvalue distribution is
shown for the smallest of the eigenvalues of $\nabla^{2}V$. Indeed,
looking at the trajectories in Figure \ref{fig:Trajectories} for
the choice $\eta(t)\equiv1$ (solid line) and $\eta(t)=1.9\sqrt{\lambda_{min}(\nabla^{2}V(\mathbf{u}(t)))}$
(curve indicated with '{*}') it is clearly seen how the choice of
damping affects the convergence. In fact, the effect is rather striking
and further analysis and tests of our conjecture is of great interest
in order to improve  DFPM.%
\begin{figure}[H]
\noindent \centering{}\includegraphics[width=8.6cm]{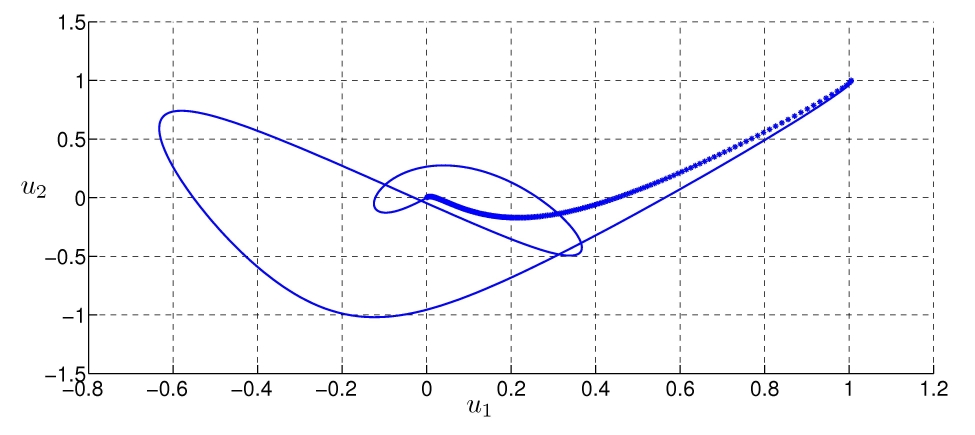}\caption{Trajectories for damping $\eta(t)\equiv1$ (solid line) and $\eta(t)=1.9\sqrt{\lambda_{min}(\nabla^{2}V(\mathbf{u}(t)))}$
(curve indicated with '{*}'). \label{fig:Trajectories} }

\end{figure}

\section{DFPM example for the Helium atom \label{sec:Helium}}

A relevant numerical application is to study the s-limit case of the
Helium ground state energy. This example is often used in atomic physics
literature as a benchmark to study numerical accuracy and efficiency.
The equation at hand is the Schrödinger equation. Due to electronic
correlation and consequently discontinuities ({}``Cato cusps'')
this is often considered to be a tough problem. Another complication
is the many-particle character leading to extremely high dimensionality.
The Helium example here only slightly touches these problems because
the full correlation term has been neglected, i.e., that term is replaced
by $1/max\left(r_{1},r_{2}\right)$ resulting in only a 2D problem.
This example is nevertherless sufficient to demonstrate many of the
properties of DFPM. More complex examples have already been tested
and DFPM remains relevant. However these fall outside the scope of
the present work.

Accordingly, consider the Schrödinger equation for the s-limit case
of Helium \citep{Edvard_cpc}:
\begin{equation}
\begin{split}
\hat{H}v\left(r_{1},r_{2}\right)&=\left[-\frac{1}{2}\frac{\partial^{2}}{\partial r_{1}^{2}}-\frac{1}{2}\frac{\partial^{2}}{\partial r_{2}^{2}}-\frac{2}{r_{1}}-\frac{2}{r_{2}}+\right. \\
&\left.+\frac{1}{max\left(r_{1},r_{2}\right)}\right]v\left(r_{1},r_{2}\right)=Ev\left(r_{1},r_{2}\right)\label{eq:test}
\end{split}
\end{equation}
The boundary conditions are given by $v\left(r_{1},0\right)=v\left(0,r_{2}\right)=v\left(R,r_{2}\right)=v\left(r_{1},R\right)=0$.
The discretization can be made by using central finite differences
with equidistant mesh sizes $h=0.1/1.1^{k}$ (for both $\Delta r_{1}$
and $\Delta r_{2}$) where $k$ is an integer chosen to get different
problem sizes. The discretized version of $\hat{H}v\left(r_{1i},r_{2j}\right)$
of a certain particle $p_{ij}$ at the position $\left(r_{1i},r_{2j}\right)$
becomes

\begin{equation*}
\begin{split}
\hat{H}v\left(r_{1i},r_{2j}\right)\approx Hu_{ij}&
=-\frac{1}{2}\frac{u_{i-1,j}+u_{i+1,j}+u_{i,j-1}+u_{i,j+1}-4u_{ij}}{h^{2}}-\\
&-\frac{2u_{ij}}{r_{1i}}-\frac{2u_{ij}}{r_{2j}}+\frac{u_{ij}}{max\left(r_{1i},r_{2j}\right)}
\end{split}
\end{equation*}

Note that the matrix $H$ is never explicitly needed so the formulation
is automatically sparse. The interaction functional component, i.e.,
the {}``force'' acting on the particle $p_{ij}$ at the position
$\left(r_{1i},r_{2j}\right)$ is given by $F_{ij}=<u|H|u>u_{ij}-Hu_{ij}$.
This can be compared with the equation (\ref{eq:eigDFPM}) derived
earlier. The notation $<u|H|u>$ is the trapezoidal approximation
to $\int_{0}^{R}v(\hat{H}v)\, dr_{1}dr_{2}$. The required norm $\left\Vert v\right\Vert =1$
is in the present context given by $<u|u>=1$. The DFPM equation \ref{DFPM}
for particle $p_{ij}$ is thus given by

\begin{equation}
<u|H|u>u_{ij}-Hu_{ij}=\mu\ddot{u}_{ij}+\eta\dot{u}_{ij},\,<u|u>=1\label{eq:test2}\end{equation}

In this case we apply constant mass and damping parameters ($\mu=1$
and $\eta=1.54$). The boundary is set to $R=15$ which is sufficient
for accurate ground state results. The integration method used to
solve the ODE (\ref{eq:test2}) is the symplectic Euler method, see
e.g. \citep{Hairer}. The related Störmer-Verlet method was tested
as well but gave no performance advantage. The test results are tabulated
in Table 6.1.\begin{table*} \centering \footnotesize \caption{Efficiency of DFPM for the s-limit Helium $^1S$ groundstate.}
\begin{tabular}{c c c c c c c c}  \hline  \hline $k$ & $N$ & $\Delta t$ & $E_0$ & DFPM (s) & DACG (s) & ARPACK (s) & S-Power (s)\\  \hline 4 & 23871 & 0.066 & -2.863893321606(6) & 0.6 & 0.9 & 7.7 & 14\\   \hline 6 & 34980 & 0.055 & -2.868655504822(7) & 1.1 & 1.5 & 17 & 34\\   \hline 8 & 51360 & 0.045 & -2.871926990227(9) & 1.9 & 2.6 & 31 & 68\\   \hline 10 & 75466 & 0.037 & -2.874170330715(4) & 3.4 & 4.7 & 62 & 153\\   \hline 12 & 110215 & 0.031 & -2.875706726414(1) & 5.8 & 8.2 & 127 & 319\\  \hline 14 & 161596 & 0.026 & -2.876758055924(6) & 10 & 14.4 & 265 & 676\\  \hline 16 & 237016 & 0.021 & -2.877477040659(9) & 18 & 24.9 & 583 & 1438\\  \hline 18 & 346528 & 0.017 & -2.877968543434(3) & 33 & 38.6 & 1188 & 3142\\  \hline 20 & 508536 & 0.014 & -2.878304445684(1) & 58 & 78.3 & 2560 & 6707\\ \hline 22 & 744810 & 0.011 & -2.878533964475(9) & 103 & 141.3 & - & -\\  \hline 24 & 1090026 & 0.009 & -2.878690772322(2) & 182 & 249.7 & - & -\\  \hline $\infty$ & $\infty$ & - & -2.8790287673(2) & - & - & - & - \\ \hline \hline \end{tabular} \end{table*}

The first column shows \emph{k} which determines the discretization
\emph{h} as mentioned above. In the second column, the corresponding
total number of particles, $N$, is listed. Only a triangular domain
needs to be computed due to even symmetry of the solution (i.e., $v\left(r_{1},r_{2}\right)=v\left(r_{2},r_{1}\right)$).
Then the third column contains the maximum timestep $\Delta t$ used
in the Symplectic Euler method (depends on \emph{h}). In the fourth
column the eigenvalues $E_{0}$ to 13 significant figures are listed.
These values can easily be extrapolated to continuum (i.e., $N\rightarrow\infty$).
This extrapolated value is listed in the last line. In the final three
columns we list the total CPU times in seconds to complete the computations
to the desired accuracy. DFPM and three other methods are compared.

A single C-code was written where the only difference was whether
a function call was made to DFPM, DACG, ARPACK or S-Power. The DACG
method (Deflation Accelerated Conjugated Gradient) is an advanced
method to find some of the lower eigenvalues of a symmetric positive
definite matrix. The algorithm is described in refs. \citep{Bergamaschi,Bergamaschi2}.
Unfortunately the DACG algorithm cannot be used directly because there
are both negative and positive eigenvalues present in the eigenvalue
problem (\ref{eq:test}). Consequently, a shift of the potential (diagonal
elements) had to be done. The shift for best performance was identified
to be 3.9. The shifted power method (i.e., 'S-Power') is a simple
method that converges to the dominant eigenvalue \citep{website:power}.
This shift was also adjusted to get the maximum performance possible
(about $0.5\left(E_{max}-E_{0}\right)$). In the case of ARPACK, it
is available at \citep{website:arpack}. This iterative numerical library
is based on the Implicitly Restarted Lanczos Method (IRLM). All tests
were performed on a Linux PC with 1 GB primary memory and the CPU
was a AMD Sempron 3600+ 2.0GHz. The compilers used were gcc and g77
version 3.4.4. Both the numerical library and C-code were compiled
using the optimization: '-O'. The CPU times were measured using the
Linux command 'time' three times (the best reading is listed in Table
6.1).

As can be seen in Table 6.1, the performance of DFPM is quite good
considering that no real optimization of DFPM has been attempted yet.
DACG performs in parity with DFPM, although it is some 40\% less efficient
for this example. Both ARPACK and S-Power are far behind. However,
ARPACK is seen to be more than twice as efficient as the S-Power method.
For the smaller problems in Table 6.1, DFPM is one order of magnitude
faster than ARPACK. For the larger problems, the advantage to DFPM
is seen to approach two orders of magnitudes. Also LAPACK was put
to test (routine DSBEVX). This routine computes selected eigenvalues
and, optionally, eigenvectors of a real symmetric band matrix. The
parameter 'jobz' was set to compute eigenvalues only and the range
was set to compute only the lowest eigenvalue. Unfortunately, due
to how DSBEVX handles matrix memory allocation, large sizes, as applied
here, is not realistic to compute. Even the smallest size \emph{N}
in Table 6.1 requires $\sim1000\, s$ to complete.

There are several reasons for good efficiency of DFPM. Firstly, the
symplectic integration method for the second order differential equation
allows a relatively large timestep. High numerical accuracy is not
necessary during the evolution towards the stationary state. Only
stability and speed is desired. Secondly, the cost per iteration is
small because the symplectic integration method is as fast as Euler's
method. Due to the damped dynamical system, the convergence rate is
exponential in time. This is also theoretically consistent with Section
\ref{sec:Convergence-rate}. %
\begin{figure}[ht]
\centering
\scalebox{0.35}{\includegraphics{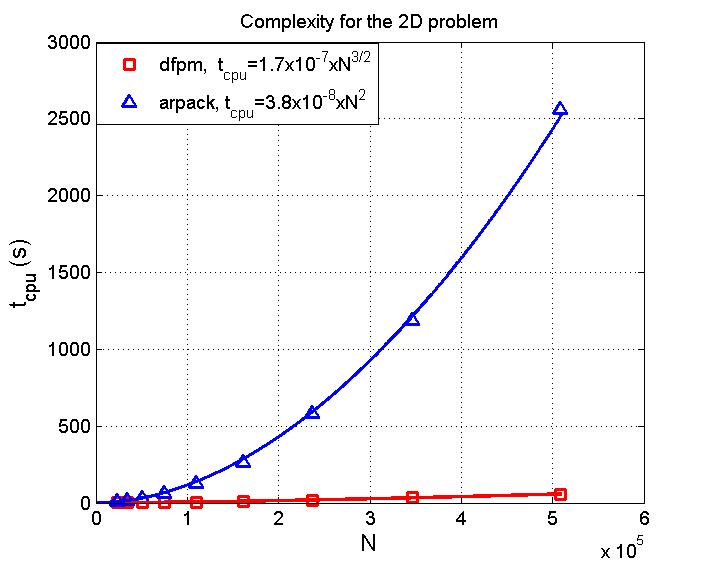}}
\caption{Computational complexity of ARPACK and DFPM.}
\end{figure}

Further, in Fig. 6.1 it is seen that the computational complexity
of ARPACK is given by $\mathcal{O}\left(N^{2}\right)$. The DFPM complexity,
however, is found to be $\mathcal{O}\left(N^{\frac{3}{2}}\right)$.
The cost of one iteration is the same for all the methods. Because
of the sparsity of $H$, it is $\mathcal{O}\left(N\right)$. The number
of iterations is what separates the various algorithms. For ARPACK
and S-Power the number of iterations are both $\mathcal{O}\left(N\right)$.
In fact, in the case of S-Power it is straight forward to prove that
the number of iterations for a \emph{d}-dimensional problem is given
by $\mathcal{O}\left(N^{2/d}\right)$, i.e., $\mathcal{O}\left(N\right)$
for $d=2$. DFPM and DACG, however, both show complexity $\mathcal{O}\left(N^{\frac{1}{2}}\right)$
for the number of iterations. It is interesting to discuss this further
for the DFPM algorithm.

Consider equation (\ref{eq:test2}) and let us assume that after only
a few iterations, $<u|H|u>\approx E_{0}$. Most of the iterations
are then spent solving:

\[
E_{0}u_{ij}-Hu_{ij}=\mu\ddot{u}_{ij}+\eta\dot{u}_{ij}\]

for each one of the particles $p_{ij}$. The symplectic integration
method is applied to approximate $u_{ij}\left(t\right)$ on its way
to the stationary solution. By applying a linear stability analysis
of the symplectic method for the problem at hand, one finds that the
maximum step size is

\[
\triangle t_{max}=\frac{2}{\sqrt{E_{1}-E_{0}}+\sqrt{E_{N-1}-E_{0}}}\]

where $E_{N-1}$ is the dominant eigenvalue. For the present problem,
using the central difference formula, it is easy to show that this
eigenvalue depends on the mesh size \emph{h} according to, $E_{N-1}=\mathcal{O}\left(1/h^{2}\right)$.
Since the mesh size only constitutes a minor correction to the lowest
eigenvalues $E_{0}$ and $E_{1}$, we have that $\triangle t_{max}=\mathcal{O}\left(h\right)$.
The equidistant mesh size \emph{h} and the total number of particles
\emph{N} in a \emph{d}-dimensional problem are related, i.e. $h\sim N^{-1/d}$,
thus $\triangle t_{max}=\mathcal{O}\left(N^{-1/d}\right)$. The number
of iterations is therefore given by $t/\triangle t_{max}=\mathcal{O}\left(N^{1/d}\right)$,
where $t$ is the total evolution time until the stationary solution
is achieved. The time $t$ does not depend on \emph{N. }That is, it
is assumed that the mesh is fine enough and that the symplectic Euler
algorithm approximately follows the true solution curve (i.e., the
continuum case). The total complexity for DFPM is then given by $\mathcal{O}\left(N^{(d+1)/d}\right)$.
In the present test case, \emph{d}=2, and the behavior in Fig. 6.1
is thus explained.

The presented benchmark indicates that DFPM is some 40\% more efficient
than DACG. DFPM can be further optimized by 1. allowing a varying
timestep during iterations, 2. preconditioning of the functional \emph{F}
and 3. allowing variation in the damping parameter during the iterations.
However, also DACG can be optimized by a preconditioning of the matrix
\emph{H}. In the benchmark tests we experience that DFPM is more robust
than DACG. DACG is quite sensitive to the selected shift. If the shift
is small it has tendencies to diverge (despite that all eigenvalues
are positive). If the shift is larger the convergence rate starts
to suffer. The selected shift is $\left|E_{0}\right|+1$ which gave
the best convergence rate. Since DACG requires that the matrix \emph{H}
is positive definite, it would seem that DFPM is a better choice for
Schrödinger problems where there often is a mix of positive and negative
eigenvalues. Initially, the lowest (negative) eigenvalue is of course
unknown which is why it can be difficult to apply DACG since one cannot
assume \emph{a priori} to know an appropriate shift. Another advantage
is that the DFPM algorithm (\ref{DFPM}) remains the same also for
nonlinear problems. The idea presented here is therefore quite versatile.
The basic idea of DFPM as a dynamical system may also be attractive
from a user's perspective since the algorithm is physically intuitive
and very pedagogical.

\section{Conclusions and Future Work\label{sec:Conclusions-and-Future}}

We have presented a new versatile method to solve equations that we
call the Dynamical Functional Particle method, DFPM. It is based on
the idea of formulating the equations as a finite dimensional damped
dynamical system where the stationary points are the solution to the
equations. The attractive properties of the method is that it has
an exponential convergence rate, includes a sparse formulation for
linear systems of equations and linear eigenvalue problems, and does
not require an explicit solver for nonlinear problems.

There are still a lot of interesting questions to be addressed. This
includes the details for solving the ODE (\ref{DFPM}) in the most
stable and efficient way. Motivated by the numerical tests reported
in Section \ref{sec:Helium} we believe that symplectic solvers, see
\citep{As08}, will be of great importance here. However, the stability
properties of the ODE solver is linked to the choice of parameters
and especially the damping parameter. The key question is how to find
the maximum time step retaining stability and getting a fast exponential
convergence rate. We are currently working on these issues using the
ideas presented here.

DFPM has proved useful for very large sparse linear eigenvalue problems
as indicated in Section \ref{sec:Helium}. It is plausible that DFPM
will also be efficient for very large and sparse linear system of
equations.

Since DFPM has a local exponential convergence rate it may be that
it can be an alternative to the standard methods for nonlinear system
of equations such as quasi-Newton or trust-region methods, see \citep{De04}.
Moreover, if there exists a potential (or a Lyapunov function) DFPM
has  global convergence properties that can be useful for developing
a solver for nonlinear problems with multiple solutions.

\bibliographystyle{plain}
\bibliography{DFPM_AO.bbl}

\end{document}